\documentclass[pdflatex,sn-mathphys-num]{sn-jnl}\usepackage{amsthm}
\usepackage{amssymb}
\usepackage{amsmath}
\usepackage{amscd}
\usepackage{color}
\usepackage{xcolor}
\usepackage{ulem}

\usepackage{graphicx}%
\usepackage{multirow}%
\usepackage{amsmath,amssymb,amsfonts}%
\usepackage{amsthm}%
\usepackage{mathrsfs}%
\usepackage[title]{appendix}%
\usepackage{xcolor}%
\usepackage{textcomp}%
\usepackage{manyfoot}%
\usepackage{booktabs}%
\usepackage{algorithm}%
\usepackage{algorithmicx}%
\usepackage{algpseudocode}%
\usepackage{listings}%

\usepackage{tikz}
\usetikzlibrary{positioning}

\newtheorem{thm}{Theorem}[section]
\newtheorem{coro}[thm]{Corollary}
\newtheorem{lemma}[thm]{Lemma}
\newtheorem{prop}[thm]{Proposition}

\newtheorem{definition}[thm]{Definition}

\newtheorem{remark}[thm]{Remark}

\newtheorem{conjecture}[thm]{Conjecture}

\makeatletter
 \renewcommand{\theequation}{%
      \thesection.\arabic{equation}}
 \@addtoreset{equation}{section}
\makeatother
\def\spmapright#1{\smash{%
 \mathop{\hbox to 1.3cm{\rightarrowfill}}
  \limits^{#1}}}

\begin{document}

\title[Propagation phenomena for operator-valued weighted shifts]{\bf Propagation Phenomena for \\ Operator-Valued Weighted Shifts}

\author*[1]{\fnm{Ra\'ul} \sur{Curto}}\email{raul-curto@uiowa.edu}
\equalcont{All authors contributed equally to this work.}

\author[2]{\fnm{Abderrazzak} \sur{Ech-charyfy}}\email{abderrazzak\_echcharyfy@um5.ac.ma}
\equalcont{All authors contributed equally to this work.}

\author[3]{\fnm{Hamza} \sur{El Azhar}}\email{el-azhar.h@ucd.ac.ma}
\equalcont{All authors contributed equally to this work.}

\author[2]{\fnm{El Hassan} \sur{Zerouali}}\email{elhassan.zerouali@fsr.um5.ac.ma}

\affil[1]{\orgdiv{Department of Mathematics}, \orgname{The University of Iowa}, \orgaddress{\street{} \city{Iowa City}, \postcode{} \state{Iowa}, \country{USA}}}

\affil[2]{\orgdiv{Laboratory of Mathematical Analysis and Applications, Faculty of Sciences, Rabat, Morocco}, \orgname{ Mohammed V University in Rabat}, \orgaddress{\street{} \city{Rabat}, \postcode{} \state{} \country{Morocco}}}

\affil[3]{\orgdiv{Laboratoire d’Informatique, Math\'ematique et leurs Applications (LIMA), Faculty of Sciences}, \orgname{Chouaib Doukkali University}, \orgaddress{\street{} \city{El Jadida}, \postcode{} \state{} \country{Morocco}}}
\equalcont{All authors contributed equally to this work.}


\abstract{This paper is devoted to the study of propagation phenomena for $2$--hyponormal, quadratically hyponormal, and cubically hyponormal operator-valued weighted shifts. \ First, we show that every {\it quadratically} hyponormal matrix-valued weighted shift with two equal weights ({\it excluding the initial weight}) is flat. \ Second, we show that a {\it cubically} hyponormal operator-valued weighted shift with two equal weights ({\it possibly including the initial weight}) is flat. \ Next, we introduce a {\it local flatness} notion for matrix-valued weighted shifts. \ We prove that $2$--hyponormal (in particular, subnormal) matrix-valued weighted shifts satisfy this stronger propagation phenomenon. \ As a result, we prove a {\it structural decomposition theorem} for $2$--hyponormal matrix-valued weighted shifts.}
\maketitle


\section{Introduction}

Let $\mathcal{H}$ be a (complex, separable, infinite-dimensional) Hilbert space, and let $\mathbf{B}(\mathcal{H})$ denote the algebra of all bounded operators on $\mathcal{H}$. \ For $S, T \in \mathbf{B}(\mathcal{H})$, we define the commutator of $S$ and $T$ by $[T, S] := TS - ST$. \ An operator $T \in \mathbf{B}(\mathcal{H})$ is said to be \textit{normal} if $[T^*, T] = 0$ and \textit{hyponormal} if $[T^*, T] \geq 0$, where $T^*$ denotes the adjoint of $T$. \ A normal operator is \textit{self-adjoint} if $T = T^*$ and is \textit{positive} if $\langle Tx, x \rangle \geq 0$ for every $x \in \mathcal{H}$. \ An operator $T$ is \textit{subnormal} if $T = N_{|\mathcal{H}}$, where $N$ is a normal operator on some Hilbert space $\mathcal{K} \supseteq \mathcal{H}$.

Subnormality and hyponormality were introduced by Paul R. Halmos in \cite{Hal}. \ The notion of hyponormality reflects the geometric aspects of normality, with implications for positive matrices. \ On the other hand, subnormality is intimately related to analyticity in complex functions, particularly through the restriction of the functional calculus to invariant subspaces. \ The class of subnormal operators, and, more broadly, the concept of subnormality, has attracted significant attention from many authors. \ Providing a complete bibliography would be too ambitious; instead, we refer to \cite{Con} for a comprehensive treatment of subnormal operators, which includes an extensive survey of subnormal \textit{scalar} weighted shifts (see Section~2 for definitions). 

In the following proposition, we assemble some known characterizations of subnormal operators needed in the sequel.

\begin{thm}\label{s-shift}\cite[Theorem 1.9]{Con} 
Let $T \in \mathbf{B}(\mathcal{H})$. \ Then the following statements are equivalent:
\begin{enumerate}
    \item $T$ is subnormal;
		
    \item (J. Bram \cite{Bra}) \ The operator-valued matrix $([T^{*i}, T^j])_{i,j \geq 0}$ is positive;
    
		\item (J. Bram - P.R. Halmos \cite{Bra})  The operator-valued matrix $(T^{*i}T^j)_{i,j \geq 0}$ is positive;
    
		\item (M. Embry \cite{Embry}) \ The operator-valued matrix $(T^{* i+j}T^{i+j})_{i,j \geq 0}$ is positive;
    
		\item (M. Embry \cite{Embry}) \ There exists an operator-valued measure $E$ supported on a compact set $K = [0, \|T\|^2]$ such that 
    \[
    T^{*n}T^n = \int_K t^n \, dE(t), \quad \text{for every } n \geq 0.
    \]
\end{enumerate}
\end{thm}

Subnormal weighted shift operators are extensively studied in functional analysis because of their interesting properties and wide range of applications. \ These operators appear naturally in various areas, including signal processing, quantum mechanics, and the study of bounded linear operators on Hilbert spaces. \ They serve as fundamental examples in operator theory and provide information on the structure of operators acting on infinite-dimensional spaces.

To introduce scalar weighted shift operators, we endow the Hilbert space $\mathcal{H}$ with a canonical orthonormal basis $\{e_n\}_{n \in \mathbb{Z}_+}$. \ The \textit{unilateral (forward) weighted shift} is the linear operator $W_{\alpha}$ defined on the basis of $\mathcal{H}$ by 
\[
W_\alpha e_n := \alpha_n e_{n+1},
\]
where $\alpha = (\alpha_n)_{n \geq 0}$ is a given sequence of positive real numbers (called \textit{weights}). 

It is well known that $W_{\alpha}$ is bounded if and only if the weight sequence is bounded. \ In this case, we have
\[
\|W_{\alpha}\| = \sup_{n \geq 0} \alpha_n < +\infty.
\]

A well-known propagation result of J. Stampfli states that a subnormal scalar weighted shift with two consecutive equal weights is flat; see \cite[Theorem 6]{Sta}. \ Recall that a scalar weighted shift $W_\alpha$ is said to be \textit{flat} if 
\[
\alpha_n = \alpha_1 \quad \text{for every } n \geq 1.
\]

We refer to \cite{Con} and \cite{shield} for an exhaustive study of scalar weighted shift operators.

Clearly, $W_{\alpha}$ is never normal. \ Subnormality for weighted shifts can be read directly from the action on the canonical orthonormal basis. \ More precisely, to describe subnormal weighted shifts, associate with $W_\alpha$ the sequence $(\gamma_n)_{n \geq 0}$, called the moment sequence of $W_\alpha$, given by 
\[
\gamma_0 := 1 \quad \text{and} \quad \gamma_k \equiv \gamma_k(\alpha) := \alpha_0^2 \alpha_1^2 \cdots \alpha_{k-1}^2 \quad \text{for } k \geq 1.
\]

We have the following formulation of the subnormality of $W_{\alpha}$, as given in Theorem \ref{s-shift}, in terms of its moment sequence.

\begin{thm}\label{ss-shift}
Let $W_{\alpha}$ be a weighted shift. \ The following statements are equivalent:
\begin{enumerate}
    \item $W_{\alpha}$ is subnormal;

    \item (Bram-Halmos) \ $(\gamma_{i+j})_{i,j \geq 0} \geq 0$ and $(\gamma_{i+j+1})_{i,j \geq 0} \geq 0$;
    
		\item (Berger; Gellar-Wallen) \ There exists a positive Borel measure $\mu$ supported on a compact set $K = [0, \|W_{\alpha}\|^2]$ such that 
    \[
    \gamma_k = \int_K t^k \, d\mu(t) \quad \text{for every } k \geq 0.
    \]
\end{enumerate}
\end{thm}

Several bridges between subnormal and hyponormal operators are introduced and studied in \cite{curtobridge}. \ These bridges or staircases are based on two families of operators: the $k$--hyponormal operators and the weakly $k$--hyponormal operators; for contractions, an alternative formulation is given in terms of $n$--contractivity (see \cite{Exner}. \ The case of scalar weighted shifts is investigated in \cite{CF2, CF1}. \ In particular, the authors studied a concrete criterion for distinguishing between subnormal, $k$--hyponormal, and weakly $k$--hyponormal scalar weighted shifts on a Hilbert space. \ More precisely:

\begin{definition}\cite[Definition 1.3]{curtobridge}
Let $k \in \mathbb{N}$, and let ${\bf T}(k) = (T_1, \cdots, T_k)$ be an $k$--tuple of operators on ${\mathcal H}$. \ We say that ${\bf T}(k)$ is (strongly) hyponormal if the operator matrix $([T_{j}^*, T_i])_{1\le i,j\le k}$ is positive.
\end{definition}

\begin{definition}
Let $k \in \mathbb{N}$. \ An operator $T$ is $k$--hyponormal if the $k$--tuple $(T, \cdots, T^k)$ is hyponormal, that is, if $([T^{*j}, T^i])_{1\le i,j\le k}$ is positive. 
\end{definition}
It is easy to see that
\begin{center}
$(k+1)$--\textrm{hyponormal} $\Rightarrow$ $k$--\textrm{hyponormal} $\Rightarrow$ 1--\textrm{hyponormal} $\Leftrightarrow$ \textrm{hyponormal}.
\end{center}

\begin{remark}{\cite[Theorem 5.1]{positivity}}
An operator matrix 
\(
\begin{pmatrix}
A & B^* \\ 
B & C
\end{pmatrix}
\)
(with \( A \) invertible) is positive if and only if \( A \geq 0 \), \( C \geq 0 \), and \( C\ge B^* A^{-1} B.  \)
\end{remark}
Another characterization of $k$--hyponormality is given as follows. \ 
An operator $T \in \mathbf{B}\mathcal{(H)}$ is $k$--hyponormal if and only if the operator matrix $(T^{*j}T^i)_{0\le i,j\le k}$ is positive. \ 
Clearly, the Bram-Halmos criterion says
$$
T \text{ is subnormal} \Leftrightarrow T \text{ is } k-\text{hyponormal for all } k \in \mathbb{N}.
$$

For convenience, let us introduce the class of  $k_E$--hyponormal operators (i.e., Embry $k$--hyponormal operators) as those for which the operator matrix $(T^{* i+j}T^{i+j})_{0\le i,j\le k}$ is positive. \ Again, it is clear that, by the Embry  criterion,    
$T$ is subnormal if and only if $T$ is $k_E$--hyponormal for all
\ $k\in\mathbb{N}$. \ Moreover,
\begin{center}
$(k+1)_E$--\textrm{hyponormal} $\Rightarrow k_E$--\textrm{hyponormal}
$\Rightarrow 1_E$--\textrm{hyponormal} $\Leftrightarrow$
\textrm{hyponormal}.
\end{center}
In particular, $k_E$--hyponormal operators also climb the staircase between hyponormal and subnormal operators.

It is established in \cite{mcc} that the  $k$--hyponormality implies $k_E$--hyponormality, thanks to the following factorization: \vspace{-12pt}
\begin{center}
\small{$$
\left( \begin{matrix}
    I   & \cdots & T^{*k}T^k \\  
     \vdots & \ddots & \vdots \\
    T^{*k}T^k & \cdots & T^{*2k}T^{2k} \\
\end{matrix}\right) =   
\left( \begin{matrix}
    I  & \cdots & 0 \\
      \vdots & \ddots & \vdots \\
     0 & \cdots & T^{*k} \\
\end{matrix}\right)
\left( \begin{matrix}
    I &   \cdots & T^{*k}\\
      \vdots & \ddots & \vdots \\
    T^k   & \cdots & T^{*k}T^k \\
\end{matrix}\right)
\left( \begin{matrix}
    I& \cdots & 0 \\
  \vdots & \ddots & \vdots \\
     0 & \cdots & T^k \\
\end{matrix}\right).$$}
\end{center}
Although the two notions are equivalent when $T$ is invertible, in the general case of operators, the reverse is not true (see \cite[Example 2.1]{mcc}). \ However, for scalar weighted shifts, the two notions are equivalent, as shown in \cite{mcc}.

\begin{thm}\cite[Theorem 2.2]{mcc}\label{mcct}
Let $W_\alpha$ be a unilateral weighted shift and $k\in\mathbb{N}$. \ Then
$$
W_\alpha \text{ is } k\text{-hyponormal} \iff W_\alpha \text{ is } k_E\text{-hyponormal}.
$$
\end{thm}
\section{$k$--hyponormal operator weighted shifts}
In this section, we focus on operator-valued weighted shifts. \ Recall that an \textit{operator matrix} $(T_{ij})_{i,j \geq 0}$, whose entries $T_{ij}$ are bounded operators such that $T_{ij} = T_{ji}^*$ for every $i, j \geq 0$, is said to be \textit{positive} if 
\[
\sum_{i,j=0}^k \langle T_{ij}x_i, x_j \rangle \geq 0, \quad \text{for every } x_0, \dots, x_k \in \mathcal{H}, \text{ and } k \in \mathbb{Z}_+.
\]

Let $A = \{A_n\}_{n=0}^\infty$ be a sequence of positive bounded operators on ${\mathcal H}$ such that  $\displaystyle\sup_{n\in\mathbb{Z}_+}\|A_n\| < \infty$. \
Let 
\[
\ell^{2}({\mathcal H}) := \left\{ (x_n)_{n\geq 0} \mid \forall n \geq 0, x_n \in {\mathcal H} \text{ and } \sum_{n=0}^\infty \|x_n\|^2 < \infty \right\}
\]
be equipped with the inner product defined as 
\[
\langle x, y \rangle := \sum_{n=0}^\infty \langle x_n, y_n \rangle
\]
where \(x = (x_n)_{n\in\mathbb{Z}_+}\) and \(y = (y_n)_{n\in \mathbb{Z}_+}\) belong to \(\ell^{2}({\mathcal H})\).

The operator-valued weighted shift associated with an invertible operator weight sequence ${\mathcal A}=\{A_n\}_{n\in\mathbb{Z}_+}$ is a bounded linear operator on 
$\ell^{2}({\mathcal H})$ given by
\[
W_{\mathcal A}(x_0, x_1, \ldots) = (0, A_0x_0, A_1x_1, \ldots),
\]
with $\|W_{\mathcal A}\| = \displaystyle\sup_{n\in\mathbb{Z}_+}\|A_n\|$.

An easy verification shows that the adjoint operator of $W_{\mathcal A}$ is the backward shift operator defined as 
\[
W_{\mathcal A}^*(x_0, x_1, \ldots) = (  A^*_0x_1, A^*_1x_2, \ldots).
\]
Computing $[W_{\mathcal A}^* ,  W_{\mathcal A}] $, we obtain
\[
[W_{\mathcal A}^* ,  W_{\mathcal A}] (x_0, x_1, \ldots) = (  A_0^2x_0, (A_1^2- A_0^2)x_1, \ldots,(A_n^2- A_{n-1}^2)x_n,\ldots),
\]
and thus $W_{\mathcal A}$ is hyponormal if and only if $A_{n}^{2}\le A_{n+1}^{2}$ for every $n\in\mathbb{Z}_+.$

Operator-valued weighted shift operators have been considered by various authors (see, for example, \cite{Fakh, ghatage, Lam}; a related Stampfli theorem has recently been proved by the authors in \cite{EC24}.

Our first result extends Theorem~\ref{mcct} to the case of operator-valued weighted shifts. \ The proof is a modification of the one in \cite[Theorem~2.2]{mcc}.

\begin{thm}\label{mccto} Let $W_{\mathcal A}$ be a unilateral operator-valued weighted shift and $k\in\mathbb{N}^*$. \ Then 
 $$W_{\mathcal A} \ is \  k-hyponormal \ \iff  W_{\mathcal A} \ is  \ k_E-hyponormal.$$
\end{thm}
\begin{proof}
 As observed before, it suffices to show that $W_{\mathcal A} \ is \  k_E-hyponormal \ \Rightarrow  W_{\mathcal A} \ is  \ k-hyponormal.$ \ We set 
$$
\mathbf{A}:= (W_{\mathcal A}^{*j}W_{\mathcal A}^i)_{0\le i,j\le k}
$$
and
$$
\mathbf{B}= (W_{\mathcal A}^{*i+j}W_{\mathcal A}^{i+j})_{0\le i,j\le k}.
$$
Let $(x_0,\cdots,x_k)\in \ell^{2}({\mathcal H})^{k+1}.$ \ Since 
$$
\ell^{2}({\mathcal H})=\ker(W_{\mathcal A}^{*p}) \oplus \overline{ W_{\mathcal A}^p(\ell^{2}({\mathcal H}))},
$$
for every $p\ge 0$, we can write 
$$
\ell^{2}({\mathcal H})^{k+1} =[0\oplus \ker(W_{\mathcal A}^{*})\oplus \cdots \oplus \ker(W_{\mathcal A}^{*k})]\bigoplus [\ell^{2}({\mathcal H})\oplus \overline{ W_{\mathcal A}(\ell^{2}({\mathcal H}))}\oplus \cdots \oplus  \overline{ W_{\mathcal A}^k(\ell^{2}({\mathcal H}))}].
$$
Let us denote 
$$
\mathcal N :=  0\oplus \ker(W_{\mathcal A}^{*})\oplus \cdots \oplus \ker(W_{\mathcal A}^{*k})
$$
and
$$
\mathcal M :=   \ell^{2}({\mathcal H})\oplus  { W_{\mathcal A}(\ell^{2}({\mathcal H}))}\oplus \cdots \oplus { W_{\mathcal A}^k(\ell^{2}({\mathcal H}))}.
$$
We clearly have 
${\mathcal N} \oplus \overline{\mathcal M}=\ell^{2}({\mathcal H})^{k+1}$. \ Moreover, ${\mathcal N}$ and $\overline{\mathcal M}$ are $\mathbf{A}$ reducing. \ Indeed, for $X = x_0\oplus W_{\mathcal A}x_1\oplus\cdots\oplus W_{\mathcal A}^kx_k \in {\mathcal M}$ we have
$$
\mathbf{A}X= \left(\sum_{j=0}^{k} W_{\mathcal A}^{*j}W_{\mathcal A}^i(W_{\mathcal A}^jx_j)\right)_{0\le i \le k}= \left(\sum_{j=0}^{k}  (W_{\mathcal A}^{*j}W_{\mathcal A}^j)W_{\mathcal A}^ix_j\right)_{0\le i \le k}\in {\mathcal M}.
$$

It follows that $\overline{\mathcal M}$ is invariant and since $\mathbf{A}$ is self-adjoint,  ${\mathcal N}$ and ${\mathcal M}$ are $\mathbf{A}$ reducing. \ As a consequence,  $\mathbf{A}\ge 0 \iff \mathbf{A}_{|{\mathcal M}} \ge 0 \mbox{ and } \mathbf{A}_{|{\mathcal N}} \ge 0$.  \  To show that $\mathbf{A}_{|{\mathcal M}} \ge 0$, let $X:= x_0\oplus W_{\mathcal A}x_1\oplus\cdots\oplus W_{\mathcal A}^kx_k \in {\mathcal M}$. 

$$  \begin{array}{ll}
    \displaystyle\langle\mathbf{A} X,X\rangle & = \displaystyle  \sum_{i,j=0}^{k} \langle(W_{\mathcal A}^{*j}W_{\mathcal A}^i)W_{\mathcal A}^jx_j, W_{\mathcal A}^ix_i\rangle = \sum_{i,j=0}^{k} \langle W_{\mathcal A}^{*i+j}W_{\mathcal A}^{i+j}x_j,x_i\rangle  \\
     & = \langle \mathbf{B} X',X'\rangle\ge 0, \ \ (\mbox{with } X'=(x_0,\cdots,x_k)).
\end{array}$$ 
 Thus $ \mathbf{A}_{|{\mathcal M}} \ge 0.$  \\

To show that $\mathbf{A}_{|{\mathcal N}} \ge 0$, we start by writing 
$$
{\mathcal N}=  {\mathcal N}_0 \oplus {\mathcal N}_1\oplus \cdots \oplus {\mathcal N}_k,
$$
with
$$
{\mathcal N}_p := 0\cap H_p\oplus \ker(W_{\mathcal A}^*)\cap H_p\oplus \cdots \oplus \ker(W_{\mathcal A}^{*k})\cap {\mathcal H}_p.
$$
Here ${\mathcal H}_p$ is the subspace of vectors $x \in \ell^2({\mathcal H})$ such that each coordinate of $x$ other than the $(p+1)^{th}$--coordinate is zero. \ We also let 
$$
{\mathcal N}_{p,q} := \ker(W_{\mathcal A}^{*q})\cap{\mathcal H}_p,
$$
for $q\le k$. \ Observe that $q\le k$, and hence ${\mathcal N}_{p,q}=0$ and for $x_p\in {\mathcal N}_p$. \ We then have $x_p= W_{\mathcal A}^p\tilde{x_p}$, for some $\tilde{x_p}\in {\mathcal H}_1$. \ We also observe that, for $X\in {\mathcal N}$, we can write
$X= \tilde{x}_0\oplus \cdots \oplus  W_{\mathcal A}^k\tilde{x_k}.$ \ As before, we get $\langle \mathbf{A} X,X\rangle=   \langle \mathbf{B} X',X'\rangle\ge 0$, with 
$X'=(\tilde{x}_0,\cdots,\tilde{x}_k)$. \ Thus $ A_{|{\mathcal N}} \ge 0$. 
 \end{proof}
\begin{remark}
    The main ingredient in the previous proof is the Wandering subspace property $ \ell^{2}({\mathcal H})=\bigoplus\limits_{p=0}^\infty W_{\mathcal A}^p\ell^{2}({\mathcal H})$.  In particular, Theorem \ref{mccto} extends to the general case of operators with the wandering subspace property.
\end{remark}
The next simplification in the study of operator-valued weighted shifts goes back to Lambert in \cite{Lam}.

\begin{thm}
Let $W_{\mathcal{A}}$ be an operator-valued weighted shift associated with an invertible operator weight sequence ${\mathcal{A}} = (A_n)_{n \geq 0}$. \ Then $W_{\mathcal{A}}$ is unitarily equivalent to an operator-valued weighted shift associated with a positive operator weight sequence.
\end{thm}

For this reason, we will assume in the sequel that ${\mathcal{A}}$ is a sequence of positive invertible operators.

For the study of strong hyponormality of operator-valued weighted shifts with positive weights, as in the scalar case, we need to introduce the operator moments of $W_{\mathcal{A}}$. \ Denote $B_0 = I$ and $B_n = A_{n-1}A_{n-2}\dots A_0$ for $n \geq 1$. \ Then $(B_n^{*} B_n)_{n \in \mathbb{Z}_+}$ is called the operator moment sequence associated with $W_{\mathcal{A}}$.

By expanding $W_{\mathcal{A}}^{*k} W_{\mathcal{A}}^l$ for every $k, l \in\mathbb{Z}_+ $, and using the general criteria of subnormality, we obtain the following formulation of Theorem \ref{s-shift} in terms of the associated operator moment sequence.

\begin{thm}\label{so-shift}
Let $W_{\mathcal{A}}$ be an operator-valued weighted shift. \ The following statements are equivalent:
\begin{enumerate}
    \item $W_{\mathcal{A}}$ is subnormal;
    \item $\big([B_i^{*}, B_j]\big)_{i, j \geq 0} \geq 0$;
    \item (Bram-Halmos) $\big(B^*_{i} B_{j}\big)_{i, j \geq 0} \geq 0$;
    \item (Embry) $\big(B^*_{i+j} B_{i+j}\big)_{i, j \geq 0} \geq 0$ and $\big(B^*_{1+i+j} B_{1+i+j}\big)_{i, j \geq 0} \geq 0$.
\end{enumerate}
\end{thm}

As an immediate consequence, we have the next well-known characterization of $k$--hyponormal operator-valued weighted shifts.
\begin{thm}\label{k-hyp-shift} Let $W_{\mathcal A}$ be an operator-valued weighted shift. \ The following statements are equivalent.
\begin{enumerate}
    \item $W_{\mathcal A}$ is $k$--hyponormal;

      \item $([B^*_{m+i},B_{m+j}])_{i,j=0 }^k \ge 0$  for every $m\ge 0$; 
      \item (Bram-Halmos)  $(B^*_{m+i}B_{m+j})_{i,j=0 }^k \ge 0$  for every $m\ge 0$;
          \item $($ Embry $)$  \   $(B^*_{m+i+j}B_{m+i+j})_{i,j= 0}^k \ge 0$ for every $m\ge 0$.
\end{enumerate}
\end{thm}
We need the next consequence of Cholesky's algorithm from \cite{cmx}.
\begin{lemma}\label{lemmacmx}\cite[Lemma 1,4]{cmx}
    An operator $T\in {\mathcal L({\mathcal H}})$  is $2$--hyponormal if and only if, for every $x,y \in {\mathcal H}, $ we have
\begin{equation}\label{2hyp}
|\langle[T^{*2}, T]y,x\rangle|^2 \le  \langle[T^*, T]x,x\rangle\langle[T^{*2}, T^2]y,y\rangle
\end{equation}
\end{lemma}

\begin{remark}{\cite[Proposition 5.14]{positivity}}
In general, let $A, C \in \mathbf{B}_+(\mathcal{H})$. \ Then an operator matrix  \(
\begin{pmatrix}
A & B^* \\ 
B & C
\end{pmatrix}
\)
is positive if and only if
$$
|\langle Bx, y\rangle|^2 \leq \langle Ax,x\rangle \langle Cy,y\rangle
$$
for every  $x,y \in \mathcal{H}$. \ In particular, \ref{2hyp} is equivalent to the positivity of the  operator matrix  
$$
\begin{pmatrix} 
[T^*, T] & [T^{*2}, T] \\ 
[T,T^{*2}] & [T^{*2}, T^2] 
\end{pmatrix}.
$$
\end{remark}

In the particular case of $2$--hyponormal operator-valued weighted shifts, we derive the following theorem: 
 \begin{thm}\label{2so-shift} Let $W_{\mathcal A}$ be a operator-valued weighted shift. \
 The following statements are equivalent.
\begin{enumerate}
    \item $W_{\mathcal A}$ is $2$--hyponormal; \vspace{6pt}
		
    \item $\begin{pmatrix} [W_{\mathcal A}^*, W_{\mathcal A}] & [W_{\mathcal A}^{*2}, W_{\mathcal A}]\\ [W_{\mathcal A},W_{\mathcal A}^{*2}] & [W_{\mathcal A}^{*2}, W_{\mathcal A}^2]
        \end{pmatrix}\ge 0$; \vspace{6pt}
    
		\item  $\begin{pmatrix} A_n^2-A_{n-1}^2 & A_{n}A_{n+1}^2-A_{n-1}^2A_n\\ A_{n+1}^2A_{n}-A_nA_{n-1}^2 &   A_{n+1}A^2_{n+2}A_{n+1}-
A_{n}A^2_{n-1}A_{n} 
        \end{pmatrix}\ge 0 \quad (n \ge 0)$, \vspace{6pt}
				
				with the convention $A_{-2}= A_{-1}= 0$.
\end{enumerate}
\end{thm}
\begin{proof}
 $(1) \iff (2)$ is known, and can be found in  \cite{Cur90} for example.\\
$(2)\iff (3)$ From Lemma \ref{lemmacmx}, we  have\
     \begin{equation}\label{propag} (2)\iff |\langle [W_{\mathcal A}^{*2}, W_{\mathcal A}]Y,X\rangle|^2 \le  \langle [W_{\mathcal A}^*, W_{\mathcal A}]X,X\rangle\langle [W_{\mathcal A}^{*2}, W_{\mathcal A}^2]Y,Y\rangle,
\end{equation}  
for arbitrary $X, Y\in \ell^{2}({\mathcal H})$.  \ Since $[W_{\mathcal A}^*, W_{\mathcal A}], [W_{\mathcal A}^{*2}, W_{\mathcal A}^2]$ are diagonal and $[W_{\mathcal A}^{*2}, W_{\mathcal A}]$ is a backward shift, it follows that 
$$
(2) \iff |\langle [W_{\mathcal A}^{*2}, W_{\mathcal A}]y,x \rangle|^2 \le \langle [W_{\mathcal A}^*, W_{\mathcal A}]x ,x \rangle\langle[W_{\mathcal A}^{*2}, W_{\mathcal A}^2]y,y\rangle,
$$
for arbitrary $x  \in {\mathcal H}_n$ and $  y \in {\mathcal H}_{n+1}$ respectively. \ Here $\mathcal{H}_n = \left\{ (x\delta_{k,n})_{k \in \mathbb{N}} \mid  x\in {\mathcal H}   \right\}
$. \ 

We deduce that \vspace{-16pt}
\begin{center}
\small{ 
\begin{equation}\label{eqn}
    |\langle (A_{n}A_{n+1}^2-A_{n-1}^2A_n)y,x \rangle|^2 \le \langle (A^2_n-A^2_{n-1})x,x\rangle\langle (A_{n+1}A^2_{n+2}A_{n+1}-
A_{n}A^2_{n-1}A_{n} )y,y\rangle. \end{equation}}
\end{center}

Using Lemma \ref{lemmacmx} again, we conclude that $(2)\iff (3)$. 
\end{proof}
 
In the matricial case (where $({\mathcal H}= {\mathbb C}^p$ for some $p\ge 2$), we derive the next local forward propagation phenomenon.
\begin{thm}Let $W_{\mathcal A}$ be a $2$--hyponormal matrix-valued weighted shift and $x\in {\mathcal H}$ such that $A_nx=A_{n-1}x$ for some $n\ge 1$. \ Then $A_nx=A_1x$  for every $n\ge 1$.
\end{thm}
\begin{proof}
Without loss of generality, we may suppose that $A_nx=A_{n-1}x=x$. \ Let us show that $A_{n+1}x=x$. \ Applying Equation \ref{eqn}, we obtain $$0=\langle (A_{n}A_{n+1}^2-A_{n-1}^2A_n)y,x\rangle = \langle y,(A_{n+1}^2A_{n}-A_nA_{n-1}^2)x\rangle=\langle y, A_{n+1}^2x- x\rangle,$$ for every $y\in {\mathcal H}.$ \ It follows that $ A_{n+1}^2x=x.$ \
Now, writing $0=(A_{n+1}^2-I)x=(A_{n+1}+I)(A_{n+1}-I)x$ and using $A_{n+1}+I$ invertible, we derive that $ A_{n+1}x= x$. \\

To obtain backward propagation, suppose that $A_{n+2}x=A_{n+3}x=x$, and let us show that $A_{n+1}x=x$. \ For \( n \geq 0 \), denote  
\begin{equation}
    A(n,0) := A_n^2 \quad \textrm{and} \quad 
    A(n,k) := A_n A_{n+1} \cdots A_{n+k-1} A_{n+k}^2 A_{n+k-1} \cdots A_{n+1} A_n,
		\end{equation}
for $k \geq 1$. \ Consider the matrix
$$
\begin{pmatrix}
    I & B_1^*B_1 & B_2^*B_2\\
    B_1^*B_1& B_2^*B_2 & B_3^*B_3\\
    B_2^*B_2& B_3^*B_3 & B_4^*B_4
\end{pmatrix}\ge 0, 
$$
and let it act on arbitrary vectors $(x_k)_{k\ge 0}$, with $x_k= 0$ for $k\notin\{n,n+1,n+2\}$. \ It follows that 
$$
\begin{pmatrix}
    I &A(n,0) &A(n,1)\\
  A(n,0)&A(n,1) &A(n,2)\\
A(n,1)&A(n,2)& A(n,3)\end{pmatrix}\ge 0.
$$ 
We also need the following operator version of Smul'jan's extension theorem from 
\cite[Proposition 2.2]{CF1}.
 
\begin{lemma}\label{smuljano} Let $X,Y$ and $Z$ be complex matrices. \ The following statements are equivalent:
\begin{enumerate}
    \item $\begin{pmatrix}
        X & Y \\
        Y^* & Z
    \end{pmatrix} \ge 0$;
    \item $X, Z \ge 0$ and there exists $U \in \mathcal{L}(\mathcal{H}, \mathcal{K})$ such that $ZU = Y^*$ and $X \ge U^*ZU$.
\end{enumerate}
\end{lemma}
We use Lemma \ref{smuljano}, with $Z=\begin{pmatrix}
 A(n,1)&A(n,2)\\
 A(n,2)&  A(n, 3)\end{pmatrix}$. \ There exists $W=\begin{pmatrix}
   W_1 \\
   W_2
  \end{pmatrix}$ such that $$
  \begin{pmatrix}
 A(n,1)&A(n,2)\\
 A(n,2)&  A(n, 3)\end{pmatrix}\begin{pmatrix}
   W_1 \\
   W_2\\
  \end{pmatrix}=\begin{pmatrix}
   A(n,0) \\
 A(n,1)
   \end{pmatrix} =\begin{pmatrix}
      A_n^2  \vspace{4pt}\\
 A_nA_{n+1}^2A_n
   \end{pmatrix}$$
If we now left multiply both sides by $ \begin{pmatrix}
 (A_{n}A_{n+1})^{-1}&0 \vspace{4pt}\\
0&  (A_{n}A_{n+1})^{-1}\end{pmatrix}$, we readily obtain
$$
\begin{array}{lll}\begin{pmatrix}
   A_{n+1}A_n &  A_{n+2}^2A_{n+1}A_n \vspace{4pt}\\
  A_{n+2}^2A_{n+1}A_n&  A_{n+2}A_{n+3}^2A_{n+2}A_{n+1}A_n\end{pmatrix}\begin{pmatrix}
   W_1 \\
   W_2\\
  \end{pmatrix}&=&\begin{pmatrix}
       A_{n+1}^{-1}A_n  \vspace{4pt}\\
  A_{n+1} A_n
   \end{pmatrix}\end{array}.
	$$
   Taking adjoints, we obtain
    $$ \begin{pmatrix}
   W_1^*,
   W_2^*
  \end{pmatrix}\begin{pmatrix}
   A_nA_{n+1} &  A_nA_{n+1}A_{n+2}^2 \vspace{4pt}\\
 A_nA_{n+1}A_{n+2}^2&  A_nA_{n+1}A_{n+2}A_{n+3}^2A_{n+2}\end{pmatrix} =\begin{pmatrix}
       A_nA_{n+1}^{-1} , 
  A_nA_{n+1}  
   \end{pmatrix}.$$ \ We evaluate at $ \left(\substack{x\\x}\right)$, with $A_{n+2}x=A_{n+3}x=x$ to obtain
 $ A_nA_{n+1}x = A_nA_{n+1}^{-1}x$. \
 
Now, left multiplying by $ A_n^{-1}$ gives  $  A_{n+1}x =  A_{n+1}^{-1}x $, which also implies $ A_{n+1}x=x$, as required. 
\end{proof}
As an immediate consequence, we recover the following global forward propagation phenomenon.
\begin{coro}\cite[Theorem 5.7]{CEIZ} Let $W_{\mathcal A}$ be a $2$--hyponormal matrix-valued weighted shift such that $A_k=A_{k-1}$ for some $k\ge 1$. \ Then $A_n=A_1$  for every $n\ge 1$.
\end{coro}
We also derive the next structural result about $2$--hyponormal operators. \ We assume, without any loss of generality, that $W_{\mathcal A}$ acting on ${\ell}^2({\mathbb C}^p)$ is a matrix-valued weighted shift such that $A_1= I$.
\begin{coro} Let $W_{\mathcal A}$ be a $2$--hyponormal matrix-valued weighted shift and denote  $ E= A_0^{-1}(ker(A_2-A_1))\subset {\mathbb C}^p$. \ Then $W_{\mathcal A}= W_{\mathcal A_1}\oplus W_{\mathcal A_2}$, with $ W_{\mathcal A_1}$ defined on ${\ell}^2(E)$,  is flat, while $ W_{\mathcal A_1}$, defined on   ${\ell}^2(E^\perp)$, is  associated with a strictly increasing weight sequence.
\end{coro}

\begin{remark}
  It is easy to observe that local propagation is more general than global propagation. \ The authors proved in \cite[Theorem 4.7]{EC24} a global propagation for subnormal operator-valued weighted shifts, but the proof is also valid for $2$--hyponormal operator-valued weighted shifts. \ Therefore, it is natural to ask whether local propagation also holds, in general, in the infinite-dimensional case.
\end{remark}

\section{Cubically hyponormal operator-valued weighted shifts}  
    Recall that a bounded operator $T\in\mathbf{B}(\mathcal{H})$ is said to be cubically hyponormal if $T + \lambda T^2+\mu T^3$ is hyponormal for all complex numbers $\lambda$ and $\mu$. \ Since cubically hyponormal operators are quadratically hyponormal, it follows that a cubically hyponormal weighted shift $W_{\mathcal A}$ such that $A_n=A_{n+1}$ for some $n\ge 1$ is automatically flat. \
    In \cite{Cur90}, examples of non-flat quadratically hyponormal scalar weighted shifts with $\alpha_0=\alpha_1$ were given. \ On the other hand, it was shown in \cite{cub} that a cubically hyponormal weighted shift $W_\alpha$ such that  $\alpha_0=\alpha_1$ is necessarily flat. \ In the next result, we obtain a characterization of cubically hyponormal matrix-valued weighted shifts.
    \begin{prop}
         Let  $W_{\mathcal A}$  be a matrix-valued weighted shift. \ Then $W_{\mathcal A}$ is cubically hyponormal if and only if for every  $s,t \in {\mathbb C}$, the pentadiagonal matrix $ M(s,t)$ is nonnegative, with  
\begin{equation}\label{pocub} 
M(s,t)=: \begin{pmatrix}
    D_{0}  &R^*_{0} &S^*_{0}       & 0 &0     &0       & \cdots &0\\
     R_{0} & D_{1}  & R^*_{1}      &S^*_{1 }&0  &0        &\cdots        &0 \\
    S_{0} & R_{1 } & D_{2} &R^*_{2} &S^*_{2 }& 0& \cdots & 0  \\
    0 & S_{1} & R_{2 } & D_{3} &R^*_{3} &S^*_{3 }& \cdots  & 0 \\
    \vdots &\vdots     & \vdots &\vdots&\vdots&\vdots&\ddots& \vdots
\end{pmatrix},
\end{equation}
and where 
 $$
 \begin{array}{lll}
 D_n & = A_n^2-A_{n-1}^2 +s^2(A(n,1)-A(n-1,-1)) + t^2(A(n,2)-A(n-1,-2)) &  \vspace{4pt}\\
     R_n &= s(A_{n+1}^2A_n-A_nA_{n-1}^2)+st(A(n+1,1)A_n-A_nA(n-1,-1))&  \vspace{4pt}\\S_n &= t(A_{n+2}^2A_{n+1}A_n-A_{n+1}A_nA_{n-1}^2)&  \\
 \end{array}
 $$
 with the extended notation $A(n,-k) := A_n\cdots A_{n-(k-1)}A_{n-k}^2A_{n-(k-1)}\cdots A_{n}$.     
   \end{prop}
   \begin{proof}
        It suffices to compute $ [W_{\mathcal A}^*+sW_{\mathcal A}^{*2}+tW_{\mathcal A}^{*3}, W_{\mathcal A}+sW_{\mathcal A}^{2}+tW_{\mathcal A}^{3}](x_n) $ when $x_n=(x\delta_{i,n})_{i\ge 0}$  for arbitrary $x\in {\mathcal H}$ and ${n\ge 0}$. 
    \end{proof} 
   
\subsection{{\bf Forward propagation for cubically hyponormal matrix-valued weighted shifts}.}

A local (forward) propagation phenomenon for cubically hyponormal operator-valued weighted shifts is shown in the next theorem. \ To analyze a certain determinant, we will need the following result from \cite[Proposition 2]{schmidt}.

\begin{prop}\label{ppos}
    A polynomial function $f(x)= a_0x^3+a_1x^2+a_2x+a_3$ is positive on ${\mathbb R}_+$ if and only if one of the following cases holds:
    \begin{enumerate}
        \item[(1)] \ $a_0>0, a_1\ge 0$, $a_2\ge 0$ and $  a_3\ge 0$;
        \item[(2)] \ $a_0>0, a_3\ge 0$ and $\delta_2=: 4a_0a_2^3+4a_1^3a_3+27a_0^2a_3^2-a_1^2a_2^2-18a_0a_1a_2a_3 \ge 0$.
    \end{enumerate}
\end{prop}

\begin{thm}\label{cubl}
    Let $W_{\mathcal A}$ be a cubically hyponormal matrix-valued weighted shift such that $A_{k}x=A_{k+1}x $ for some unit vector $x\in {\mathcal H}$ and some $k\ge 1$. \ Then
     $A_nx=A_{k}x $ for every $n\ge k$.
\end{thm}
\begin{proof}
 For the forward propagation, and without any loss of generality, we take $A_0=I$. \ Suppose that $A_1x=x$ and let us show that $A_2x = x$. 

Computing  the compression of $ M(s,t)$ on ${\mathcal H}\oplus {\mathcal H}\oplus{\mathcal H}$, it follows that 
\begin{equation}\label{cube}
 \begin{pmatrix}
    D_{0}  &R^*_{0} &S^*_{0}    \\
     R_{0} & D_{1}  & R^*_{1}     \\
    S_{0} & R_{1 } & D_{2}
\end{pmatrix}\ge 0.  
\end{equation}
Using vectors of the form $(ax,bx,cx)$ in Equation \ref{cube} (with $a, b, c in \mathbb{C}$), we obtain 
$$
\begin{pmatrix}
 1+s^2  +t^2\|A_2x\|^2&  s +st\|A_2x\|^2   &   t\|A_2x\|^2    \vspace{4pt} \\
 s +st\|A_2x\|^2  &  s^2\|A_2x\|^2+t^2\|A_3A_2x\|^2 & \langle R_1^*x,x\rangle \vspace{4pt} \\
   t\|A_2x\|^2             & \langle R_1x,x\rangle      & \langle D_2x,x\rangle \\
 \end{pmatrix} \ge 0,
 $$ 
where 
$$
\langle Sx,x\rangle := \|A_2x\|^2-1 +s^2(\|A_3A_2x\|^2-1)+t^2\|A_4A_3A_2x\|^2
$$
and
$$
\langle  R_1x,x\rangle :=\langle  R_1^*x,x\rangle =s(\|A_2x\|^2-1+t\|A_3A_2x\|^2).
$$
In particular, $P(s,t)$, the determinant of the previous $3 \times 3$ matrix, is nonnegative for all $s,t\in\mathbb{R}$. \ We now expand $P(s,t)$, that is, 
$$
P(s,t)=p_0(t)s^6+p_1(t)s^4+p_2(t)s^2+p_3(t),
$$
where each $p_i(t)$ is a polynomial ($i=0, 1, 2, 3$). \ We readily obtain
$$  \begin{array}{lll}
     p_0(t)& = & \|A_2x\|^2(\|A_3A_2x\|^2-1); \vspace{4pt}\\
      p_1(t)&=&\Gamma + 2(\|A_3A_2x\|^2-2\|A_2x\|^2\|A_3A_2x\|^2+\|A_2x\|^2 )t + q_1(t)t^2; \vspace{4pt}\\
      p_2(t)& = &-2\Gamma t + \|A_2x\|^2\|A_4A_3A_2x\|^2 -\|A_2x\|^2]t^2+ q_2(t)t^4 \vspace{4pt}\\ 
      p_3(t)& =& \Gamma t^2 +q_3(t)t^4, 
\end{array}
$$
where $\Gamma := \|A_3A_2x\|^2(\|A_2x\|^2-1)$ and where  $q_1(t), q_2(t)$ and $ q_3(t)$ are polynomials.

We now observe that $P(s,t)\ge 0$ for every $s,t \in {\mathbb R}$ if and only if either (1) or (2) in Proposition \ref{ppos} is satisfied. 
\begin{itemize}
    \item In the first case, $p_2(t)\ge 0 $ for every $t\in {\mathbb R}$ implies
		$$
		\lim\limits_{t\to 0^+}\frac{p_2(t)}{t}= - 2(\|A_2x\|^2-1)\ge  0.
		$$
		Since $A_2^2\ge I$, we obtain $\|A_2x\|^2-1=0$. \vspace{6pt}
		
    \item If instead (2) holds, and we expand $\delta_2(t)=ct^3 +\delta(t)t^4$, we obtain
$$
    \begin{array}{ll}
        c=& (\|A_2x\|^2-1)^3[16\|A_2x\|^2(\|A_3A_2x\|^2-1) + 4\|A_3A_2x\|^2(\|A_2x\|^2-1)\\
         & +4\|A_2x\|^2(\|A_4A_3A_2x\|^2-1)],
    \end{array}
$$
where $\delta(t)$ is a polynomial. \ We use $c\ge 0$ and $\lim\limits_{t\to 0^+}\frac{\delta_2(t)}{t^3}= -c \ge  0$, to deduce that $c=0$. \ Now, it is clear from  $c=0  $ that $\|A_2x\|^2-1=0$.
\end{itemize} 
Since $\|A_2x\|=1$ in both cases, we conclude that $A_2x=x$.  
\end{proof}

We now establish a straightforward consequence of Theorem \ref{cubl}. 

 \begin{coro}\label{cublcor}
    Let  $W_{\mathcal A}$  be a cubically hyponormal matrix-valued weighted shift, let $x\in {\mathcal H}$, and assume that $A_{0}x=A_{1}x$. \ Then $A_nx=A_{0}x $ for every $ n\ge 0$.
 \end{coro}

\subsection{{\bf Backward propagation for cubically hyponormal matrix-valued weighted shifts}.} \ In this subsection, we first pose the following conjecture and then prove a structural result. 
\begin{conjecture}\label{cubl0}
    Let $W_{\mathcal A}$ be a cubically hyponormal matrix-valued weighted shift, and let $x\in {\mathcal H}$. \ If   $A_{k}x=A_{k+1}x$ for some $k \ge 1$, then
     $A_nx=A_{1}x$ for every $ n\ge 0$.
\end{conjecture} 
 
\begin{remark}
Although we have been unable to settle Conjecture \ref{cubl0}, we present below the key steps needed for an affirmative answer.

Assume $A_2=I$ and $A_3x=x$; we need to show that $A_1x=x$. \ Before going further, we mention that, in view of the forward propagation property (already obtained), we have $A_4x=x$. 
\ Again, taking a suitable compression, it follows that 
\begin{center}
\small{\begin{equation}\label{pocubb} 
\begin{pmatrix}
    D_{1}  & R^*_{1}      &S^*_{1 }\\
     R_{1 } & D_{2} &R^*_{2} \\
    S_{1} & R_{2 } & D_{3}
\end{pmatrix}\ge 0. 
\end{equation}}
\end{center}
Applying the positivity in (\ref{pocubb}) to vectors of the form $(aA_1x,bx,cx)$ (where $a,b,c \in \mathbb{C}$ and $x \in \mathcal{H}$), we readily obtain 
\begin{center}
 $$
 \begin{pmatrix}
   \langle D_1A_1x,A_1x\rangle & \langle R^*_1x,A_1x\rangle &   \langle S_1x,A_1x\rangle  \\
 \langle R_1A_1x,x\rangle  &  \langle D_2x,x\rangle&  \langle R_2x,x\rangle  \\
   \langle S_1A_1x,x\rangle   & \langle R_2x,x\rangle   &  \langle D_3x,x\rangle
\end{pmatrix}\ge 0.
$$
\end{center}
Here
$$
\begin{array}{ll}
  \langle D_1A_1x,A_1x\rangle    & = \|A_1^2x\|^2-\|A_0A_1x\|^2+s^2 \|A_1^2 x\|^2+t^2 \|A_3A_1^2 x\|^2, \vspace{4pt}\\
     \langle D_2x,x\rangle    & =  (1-\|A_1x\|^2)+ s^2(1-\|A_0A_1x\|^2)+t^2, \vspace{4pt}\\
     \langle D_3x,x\rangle    & = s^2(1-\|A_1x\|^2)+ t^2(1-\|A_0A_1x\|^2), \vspace{4pt}\\
     \langle R_1A_1x,x\rangle  & =  s(\|A_1x\|^2 - \|A_0A_1x\|^2)+st\|A_1x\|^2, \vspace{4pt}\\
   \langle R_2x,x\rangle  & =  s(1-\|A_1x\|^2)+st(1-\|A_0A_1x\|^2), \vspace{4pt}\\
   \langle S_1A_1x,x\rangle  & = t(\|A_1x\|^2- \|A_0A_1x\|^2)
\end{array}
$$
In particular, the determinant of the matrix above, $Q(s,t)$, is nonnegative for all $s,t \in {\mathbb R}$. \  We expand $Q(s,t)$ to obtain:
$$
Q(s,t)= b_0(t)s^6+b_1(t)s^4+b_2(t)s^2+b_3(t),
$$
where
$$
\begin{array}{lll}
     b_0(t)& = & \|A_1^2 x\|^2(1-\|A_0A_1x\|^2)(1-\|A_1x\|^2), \vspace{4pt}\\
      b_1(t)&= &-(1-\|A_1x\|^2)\Gamma -2(1-\|A_1x\|^2)\Gamma_1t+q_1(t)t^2, \vspace{4pt}\\
      b_2(t)& = & 2(1-\|A_1x\|^2)\Gamma t + (1-\|A_1x\|^2)\Gamma_2 t^2+q_2(t)t^3, \vspace{4pt}\\
      b_3(t)& =&  -(1-\|A_1x\|^2)\Gamma t^2+ \Gamma_3 t^4 +  (1-\|A_1x\|^2)\|A_3A_1^2x\|^2t^6, \vspace{4pt}\\
      \delta_2(t)
      &=& 4(1-\|A_1x\|^2)^4\Gamma^3(\|A_0A_1x\|^2-\|A_1^2x\|^2) t^3+q_4(t)t^4,
\end{array}
$$
with $q_i(t)$ polynomials and 
$$
\begin{array}{ll}
\Gamma &:= (\|A_1x\|^2-\|A_0A_1x\|^2)^2-(1-\|A_0A_1x\|^2)(\|A_1^2x\|^2-\|A_0A_1x\|^2), \vspace{4pt}\\
\Gamma_1 &:=\|A_1x\|^2(\|A_1x\|^2-\|A_0A_1x\|^2)+\|A_1^2x\|^2(1-\|A_0A_1x\|^2), \vspace{4pt}\\ 
\Gamma_2 &:= 2\|A_1x\|^2(\|A_1x\|^2-\|A_0A_1x\|^2)+\|A_1^2x\|^2(1-\|A_0A_1x\|^2) \vspace{4pt}\\
& +(\|A_1^2x\|^2-\|A_0A_1x\|^2)\ge0, \vspace{4pt}\\
\Gamma_3 & := \|A_3A_1^2x\|^2(1-\|A_0A_1x\|^2)(1-\|A_1x\|^2) -\Gamma.\\
\end{array}
$$
We assume first that $(1-\|A_1x\|^2)\Gamma \ne 0$. \ Using $b_3(t)\ge 0$ for $t$ close to zero,  we get $-(1-\|A_1x\|^2)\Gamma>0$, and hence $\Gamma<0$. \ Now, as in the forward case, if $ (1-\|A_1x\|^2)\Gamma \ne 0$, then both $b_2$ and $\delta_2$ must change sign at zero, which contradicts $Q\ge 0$.\\
Thus, we may assume that $\Gamma= 0$ and, since either $\Gamma_1=0$ and $\Gamma_2=0$ imply that $1-\|A_1x\|^2= 0$, we will also take $\Gamma_1\ne 0$ and $\Gamma_2\ne 0$. \  Let us show that $1-\|A_1x\|^2= 0$. \ Seeking a contradiction, we suppose that $1-\|A_1x\|^2\ne 0$. \ Back in the expression of $Q(s,t)$, we will have   
$$
\begin{array}{lll}
     b_0(t)& = & \|A_1^2 x\|^2(1-\|A_0A_1x\|^2)(1-\|A_1x\|^2), \vspace{4pt}\\
      b_1(t)&= &-2\Gamma_1t+q_1(t)t^2, \vspace{4pt}\\
      b_2(t)& = &  (1-\|A_1x\|^2)\Gamma_2 t^2+q_2(t)t^3, \vspace{4pt}\\
      b_3(t)& = & \Gamma_3 t^4+ (1-\|A_1x\|^2)\|A_3A_1^2x\|^2t^6, \vspace{4pt}\\
      \delta_2(t)& = & -\|A_1^2x\|^2) t^3+q_4(t)t^4.
\end{array}
$$
To settle Conjecture \ref{cubl0}, what is needed is a careful and conclusive analysis of the sign behavior of the various quantities in the previously displayed identities.
\end{remark}

We now state a structural result for cubically hyponormal operators. \ Let $W_{\mathcal A}$  be a cubically hyponormal matrix-valued weighted shift with commuting weights such that $A_1= I$. \ For $ E= A_0^{-1}(ker(A_2-I))\subset {\mathbb C}^p$ and $E_0= ker(A_0 -I)$, we now derive, using Theorem \ref{cubl}, that $E_0\subset E_1$. \  Denote $E_1=A_0^{-1}(ker(A_2-I))\ominus E_0$ and $E_2=E^\perp$. \ Then

\begin{coro} Let $W_{\mathcal A}$ be a cubically hyponormal matrix-valued weighted shift such that  $A_1=I$. \ Under the previous notation,  
$W_{\mathcal A}= W_{\mathcal A_0}\oplus W_{\mathcal A_1}\oplus W_{\mathcal A_2}$ on 
$ {\ell}^2({\mathbb C}^p)= {\ell}^2(E_0)\oplus{\ell}^2(E_1)\oplus{\ell}^2(E_2)$, with 

\begin{enumerate} 
 \item $W_{\mathcal{A}_0}$ is the matrix-valued unweighted shift;
    \item $W_{\mathcal{A}_1}$ is flat;
    \item  $W_{\mathcal{A}_2}$ is associated with a strictly increasing weight sequence.
\end{enumerate} 
\end{coro}

\begin{remark}
    Since every $3$--hyponormal operator is cubically hyponormal, the previous results apply to $3$--hyponormal matrix-valued weighted shifts and hence to subnormal matrix-valued weighted shifts.
\end{remark}

\section{Quadratically hyponormal operator-valued weighted shifts}

We recall the following definition from \cite{Cur90}.

\begin{definition}
 An operator  $T\in \mathbf{B}(\mathcal{H})$ is said to be:
 \begin{itemize}
     \item Weakly $n$--hyponormal, if for every complex polynomial $P$ with degree $n$ or less, the operator $P(T)$ is hyponormal.
 \item  Polynomially hyponormal if it is weakly $n$--hyponormal for every $n\ge 1.$
  \end{itemize}
\end{definition} 
\begin{remark} We mention the following 
    \begin{itemize}
        \item Since $P(T)$ is hyponormal if and only if $(P-P(0))(T)$ is hyponormal, we may assume without loss of generality that $P(0)=0$
        \item A weakly $2$--hyponormal $T$ is called quadratically hyponormal. \ In particular,  $T$ is   quadratically hyponormal if and only if $T+ \lambda T^2$ is hyponormal for every $\lambda \in {\mathbb C}.$
    \end{itemize}
\end{remark}
We first observe that $W_{\mathcal A}$ is quadratically hyponormal if and only if $$C_\lambda=:[(W_{\mathcal A} + \lambda W_{\mathcal A}^2)^*, W_{\mathcal A} + \lambda W_{\mathcal A}^2]\ge 0$$ for every complex  $ \lambda$. 

For $x_n=(0,\cdots,0,x,0,\cdots)\in {\mathcal H}_n=\displaystyle\bigoplus_{i=0}^{\infty} \delta_{i,n} \mathcal{H}$,
where \( \delta_{i,n} \) is the Kronecker delta, ensuring that \( \mathcal{H} \) appears only in the \( n \)-th position while all other components are zero. 
\\We obtain
$$C_\lambda x_n=(0,\cdots,0, R_{n-1}^*x,D_{n}x,  R_{n}x, 0,\cdots)\in {\mathcal H}_{n-1} \oplus {\mathcal H}_n\oplus {\mathcal H}_{n+1}, $$
with  \ $D_{n} =A_n^2- A_{n-1}^2 +|\lambda|^2(A_nA^2_{n+1}A_n-A_{n-1}A^2_{n-2}A_{n-1})$ \ and \ $R_{n}=\lambda(A^2_{n+1}A_{n}-A_{n}A^2_{n-1})$.
Now, since $C_\lambda \ge 0,$ we obtain 
\begin{equation}
 M_{n}(\lambda) =: \begin{pmatrix}
    D_{0}           &  R^*_{0}        & 0       & 0          & \cdots &0\\
     R_{0} & D_{1}& R^*_{1}      & 0        &\cdots        &0\\
     0 &R_{1} & D_{2} &R^*_{2} &\cdots & 0  \\
     \vdots &\ddots&\ddots&\ddots&\ddots&\vdots& \\    0 & 0         & \cdots        &    R_{n-2}  &D_{n-1} &R^*_{n-1}\\ 
    0 & 0         & \cdots        &    0  &R_{n-1} &D_{n}\\
\end{pmatrix}\ge 0. \label{quadr}
\end{equation}
for every $\lambda \in {\mathbb C}$ and $n\ge 0$. \ 
Thus, we have 
\begin{prop} $W_{\mathcal A}$ is quadratically hyponormal if and only if $ M_{n,\lambda} \ge 0$, for every $n\in {\mathbb N}$ and $\lambda \in {\mathbb C}.$
\end{prop}

We begin with a propagation result in the matrix-valued case.
 \begin{prop}\label{propag1} Let ${\mathcal A}=(A_n)_{n\ge 0}$ be a sequence of non negative matrices. \  
  Suppose $W_{\mathcal A}$ is quadratically hyponormal and $A_n=A_{n+1} $ for some $n\ge 1$, then either $A_{n-1}=A_n=A_{n+1} $ or $A_n=A_{n+1}=A_{n+2}$.
 \end{prop}

\begin{proof} \ Suppose $A_1=A_2= I$ and let us show that either $A_0=I$ or $A_3=I$. \ From (\ref{quadr}), we readily obtain \vspace{-12pt}
\begin{center}
    \tiny{\[ {  \begin{pmatrix}
(1+s^2)A_0^2 & sA_0 & 0&0 &0 \\
sA_0 & (1+s^2)I -A_0^2  &s( I- A^2_{0})& 0 &0\\
   0& s( I- A^2_{0}) &s^2(A^2_{3}-A_0^2 ) &s(A^2_{3}-I) & 0\\
         0&0 &        s(A^2_{3}-I) & A_3^2- I +s^2(A_3A^2_{4}A_3-I)&s( A_{3}A^2_{4} -A_{3}) \\ 
   0& 0 & 0 & s( A^2_{4}A_{3}-A_{3})&A_4^2- A_{3}^2 +s^2(A_4A^2_{5}A_4-A_{3}^2)
\end{pmatrix}\ge 0.} \]}
\end{center}Applying this inequality to vectors of the form $(a_0A_0x,a_1x, a_2x,a_3x,a_4A_3x)$ (where $a_0, a_1, a_2, a_3, a_4 \in \mathbb{C}$ and $x\in\mathcal{H}$), it follows that  
\begin{center}
\small{\begin{equation}\label{equ}
{  \begin{pmatrix}
(1+s^2)\|A_0^2x\|^2 & s\|A_0x\|^2 & 0&0 &0 \\
s\|A_0x\|^2 & \Gamma_1+ s^2\|x\|^2&  s\Gamma_1&0 &0 \\
  0&  s\Gamma_1  &s^2(\|A_3x\|^2-\|A_0x\|^ 2) &s\Gamma_2& 0  \\
 0&0& s\Gamma_2  &\Gamma_2 +s^2\Gamma_3 & s\Gamma_3'\\ 
  0&0 & 0 &s\Gamma_3' & \Gamma_3''+s^2\Gamma_4
\end{pmatrix} \ge 0,} \end{equation}}
\end{center}
where 
$$
\Gamma_1=  \|x\|^2-\|A_0x\|^2,
$$
$$
\Gamma_2=  \|A_3x\|^2-\|x\|^ 2,
$$
$$
\Gamma_3= \|A_4A_3x\|^2-\|x\|^2,
$$
$$
\Gamma_3'= \|A_4A_3x\|^2-\|A_3x\|^ 2,
$$
$$
\Gamma_3''= \|A_4A_3x\|^2-\|A_3^2x\|^ 2
$$
and
$$
\Gamma_4=   \|A_{5}A_4A_3x\|^2-\|A_{3}^2x\|^2.
$$
In particular, its determinant is positive, and using the column operation $C_3 \to C_3- s(C_2+C_4)$ gives 

\[{  \left|\begin{array}{ccccc}
(1+s^2)\|A_0^2x\|^2 & s\|A_0x\|^2 &-s^2\|A_0x\|^2&0 &0 \\
s\|A_0x\|^2 & \Gamma_1+ s^2\|x\|^2&  -s^3\|x\|^2&0 &0 \\
  0&  s\Gamma_1  & 0 &s\Gamma_2& 0  \\
 0&0& -s^3\Gamma_3 &\Gamma_2 +s^2\Gamma_3& s\Gamma_3'\\ 
  0&0 & -s^2\Gamma_3'  &s\Gamma_3' & \Gamma_3''+s^2\Gamma_4
\end{array}\right| \ge 0,} \] \vspace{6pt}
With the row operation $R_3\to R_3-s(R_2+R_4)$, we obtain
\[{ \left|\begin{array}{ccccc}
(1+s^2)\|A_0^2x\|^2 & s\|A_0x\|^2 &-s^2\|A_0x\|^2&0 &0 \\
s\|A_0x\|^2 & \Gamma_1+ s^2\|x\|^2&  -s^3\|x\|^2&0 &0 \\
 -s^2\|A_0x\|^2& -s^3\|x\|^2 & s^4\|A_4A_3x\|^2&-s^3\Gamma_3& -s^2\Gamma_3'  \\
 0&0& -s^3\Gamma_3&\Gamma_2 +s^2\Gamma_3& s\Gamma_3'\\ 
  0&0 & -s^2\Gamma_3'  &s\Gamma_3' &\Gamma_3''+s^2\Gamma_4
\end{array}\right| \ge 0} .\] \vspace{6pt}
Factoring out $s^4$, we get 
\[{   p(s)=   \left|\begin{array}{ccccc}
(1+s^2)\|A_0^2x\|^2 & s\|A_0x\|^2 &\|A_0x\|^2&0 &0 \\
s\|A_0x\|^2 & \Gamma_1+ s^2\|x\|^2&  s\|x\|^2&0 &0 \\
 \|A_0x\|^2& s\|x\|^2 & \|A_4A_3x\|^2&s\Gamma_3& \Gamma_3'  \\
 0&0& s\Gamma_3 &\Gamma_2 +s^2\Gamma_3& s\Gamma_3'\\ 
  0&0 & \Gamma_3'  &s\Gamma_3' & \Gamma_3''+s^2\Gamma_4
\end{array}\right| \ge 0.} \] \vspace{6pt}

In particular, 
\[  
    p(0)   = \left|\begin{array}{ccccc}
 \|A_0^2x\|^2 &0 &\|A_0x\|^2&0 &0 \\
0 & \Gamma_1& 0&0 &0 \\
 \|A_0x\|^2 & 0 & \|A_4A_3x\|^2 &0& \Gamma_3'  \\
 0&0& 0 &\Gamma_2 &0\\ 
  0&0 & \Gamma_3'  &0 & \Gamma_3''
\end{array}\right| =  \Gamma_1\Gamma_2\left|\begin{array}{ccc}
 \|A_0^2x\|^2  &\|A_0x\|^2&0 \\
 \|A_0x\|^2  & \|A_4A_3x\|^2 & \Gamma_3'  \\
  0 & \Gamma_3'  & \Gamma_3''
\end{array}\right| \ge 0. \vspace{6pt}
 \]
Expanding, we obtain  
$$\begin{array}{rl}
   p(0)  = &\Gamma_1\Gamma_2[\|A^2_{0}x\|^2(\|A_4A_3x\|^2\Gamma_3''- \Gamma_3^{'2}) - \|A_{0}x\|^4\Gamma_3'']\\ =& 
    \Gamma_1\Gamma_2\|A^2_{0}x\|^2(-\|A_4A_3x\|^2\|A_{3}^2x\|^2   +2\|A_4A_3x\|^2\|A_{3}x\|^2-\|A_{3}x\|^4)\\
     &  - \Gamma_1\Gamma_2\|A_{0}x\|^4(\|A_4A_3x\|^2- \|A_{3}^2x\|^2)
\end{array}$$
Assume now that $A_0\ne I$ and let us show that $A_3=I$.  \  Using the canonical decomposition for the self-adjoint matrices $\mathcal{H} = \displaystyle\bigoplus\limits_{\lambda\in \sigma(A_0)} E_{\lambda}  $ as a direct sum of eigenspaces, it suffices to show that 
$A_{3}|_{E_{\lambda}}=I|_{E_{\lambda}}$ for arbitrary $\lambda$.\\

We start with $ \lambda \ne 1$  an eigenvalue and   $x\in E_{\lambda}$   an  associated unit eigenvector. \ From $\|A^2_{0}x\|^2 = \|A_{0}x\|^4 =|\lambda|^4$, we get
$$\begin{array}{rl}
p(0)  =&|\lambda|^4\Gamma_1\Gamma_2(-\|A_4A_3x\|^2\|A_{3}^2x\|^2   + 2\|A_4A_3x\|^2\|A_{3}x\|^2\\ & -\|A_4A_3x\|^2+\|A_{3}^2x\|^2 -\|A_{3}x\|^4) \vspace{6pt}\\ 
     = &|\lambda|^4\Gamma_1\Gamma_2[-\|A_4A_3x\|^2\|A_{3}^2x-x\|^2  + \|A_{3}^2x\|^2 -\|A_{3}x\|^4] \vspace{6pt}\\
     \le & |\lambda|^4\Gamma_1\Gamma_2[-\|A_{3}^2x-x\|^2  + \|A_{3}^2x\|^2 -\|A_{3}x\|^4] \vspace{6pt}\\
    \end{array}$$
It follows that $p(0) \le 0$ an then  that  $p(0) =0$.

We deduce that $\Gamma_2= 0$ or $\|A_4A_3x\|^2\|A_{3}^2x-x\|^2  = \|A_{3}^2x\|^2 -\|A_{3}x\|^4 $.

Suppose $\|A_4A_3x\|^2\|A_{3}^2x-x\|^2  = \|A_{3}^2x\|^2 -\|A_{3}x\|^4 $. \ We will have 
$$
\begin{array}{ll}0&=\|A_4A_3x\|^2\|A_{3}^2x-x\|^2  -( \|A_{3}^2x\|^2 -\|A_{3}x\|^4)\vspace{6pt}\\
& =(\|A_4A_3x\|^2-1)\|A_{3}^2x-x\|^2 +\|A_{3}^2x-x\|^2  -( \|A_{3}^2x\|^2 -\|A_{3}x\|^4) \vspace{6pt}\\
& =(\|A_4A_3x\|^2-1)\|A_{3}^2x-x\|^2 +(\|A_{3}^2x\|^2-2\|Ax\|^2+1)  -( \|A_{3}^2x\|^2 -\|A_{3}x\|^4) \vspace{6pt}\\
& =(\|A_4A_3x\|^2-1)\|A_{3}^2x-x\|^2 +(\|A_{3}^2x\|^2-2\|Ax\|^2+1)  -( \|A_{3}^2x\|^2 -\|A_{3}x\|^4) \vspace{6pt}\\
&=(\|A_4A_3x\|^2-1)\|A_{3}^2x-x\|^2 +(\|A_{3}x\|^2-1)^2
\end{array}
$$
and hence $\|A_{3}x\|^2-1$.
 
Suppose now that  $A_0x=x$. \ Plugging in Equation \eqref{equ}, we derive that 
\begin{equation}  \begin{pmatrix}
  s^2(\|A_3x\|^2-\|A_0x\|^ 2) &s\Gamma_2& 0  \\
   s\Gamma_2  &\Gamma_2 +s^2\Gamma_3 & s\Gamma_3'\\ 
   0 &s\Gamma_3' & \Gamma_3''+s^2\Gamma_4
\end{pmatrix} \ge 0, \end{equation}
and then 
\begin{equation}\label{equ11}
Q(s) =: \left|\begin{array}{ccc}
s^2(\|A_3x\|^2-\|A_0x\|^ 2) &s\Gamma_2& 0  \\
 s\Gamma_2  &\Gamma_2 +s^2\Gamma_3& s\Gamma_3'\\ 
 0 &s\Gamma_3' &\Gamma_3''+s^2\Gamma_4
\end{array}\right| \ge 0, \end{equation} 
A computation now shows that $Q(s) = -s^4 (\|A_{3}x\|^2-1)^3\le 0.$  \ Hence  $A_3x~=~x $ as required. 
\end{proof}

We also have the following propagation result in the general case of operator-valued weighted shifts.
 \begin{prop}\label{propag2} Suppose $W_{\mathcal A}$ is quadratically hyponormal and $A_n=A_{n+1}=A_{n+2}$ for some $n\ge 0$, then $A_k=A_n$ for every $k\ge 1.$
 \end{prop}
\begin{proof}
Multiplying both sides by $A_0^{-1}$, we may suppose $A_0=I$. \
Suppose $A_0=A_1=A_2= I$ and let us show that $A_3=I$. \ Notice that 
\[ M_{n,s} \ge 0 \mbox{ for  every } n\in\mathbb{N} \mbox{ and } s\in\mathbb{R} \Rightarrow   \begin{pmatrix}
        D_{2,s} &R^*_{2,s} &0    \\
     R_{2,s}  &D_{3,s} &R^*_{3,s}\\ 
      0  &R_{3,s} &D_{4,s}\\
\end{pmatrix}\ge 0. \] \ Therefore, we readily obtain \vspace{-12pt}
\begin{center}
\tiny{$$  \begin{pmatrix}
    s^2( A^2_{3} -I) &s(A^2_{3}-I) &0  \\
                 s(A^2_{3}-I) &A_3^2- I +s^2(A_3A^2_{4}A_3-I)&s( A_{3}A^2_{4} -A_{3}) \\ 
   0 & s( A^2_{4}A_{3}-A_{3})&A_4^2- A_{3}^2 +s^2(A_4A^2_{5}A_4-A_{3}^2)
\end{pmatrix}\ge 0. $$}
\end{center}
Denote $\Gamma =\|A_{3}y\|^2-\|y\|^2$. \ Applying to $(ay,by, cA_3y)$, for arbitrary $a,b,c$, it follows that \vspace{-12pt} 
\begin{center}
\tiny{$$   \begin{pmatrix}
    s^2\Gamma  &s\Gamma  &0  \\
                 s\Gamma  &\Gamma + s^2(\|A_4A_3y\|^2-\|y\|^ 2)  & s(\|A_4A_3y\|^2-\|A_3y\|^ 2)\\ 
   0 &s(\|A_4A_3y\|^2-\|A_3y\|^ 2) &\|A_4A_3y\|^2- \|A_{3}^2y\|^2 +s^2(\|A_{5}A_4A_3y\|^2-\|A_{3}^2y\|^2) 
\end{pmatrix} \ge 0. $$}
\end{center}
 In particular, the determinant is positive. \ After the  column operation $C_2\to C_2-\frac{1}{s}C_1$, and simplification by $s^2$, we obtain \vspace{-12pt} 
 \begin{center}
\tiny{$$p(s) =: \left|\begin{array}{lcc}
  \|A_4A_3y\|^2-\|y\|^ 2 & &\|A_4A_3y\|^2-\|A_3y\|^ 2 \\ 
  \|A_4A_3y\|^2-\|A_3y\|^ 2& & \|A_4A_3y\|^2- \|A_{3}^2y\|^2 +s^2(\|A_{5}A_4A_3y\|^2-\|A_{3}^2y\|^2) 
\end{array}\right| \ge 0, $$ }
\end{center}
and then 
$$ \begin{array}{lll}
     p(0)&= (\|A_4A_3y\|^2-\|y\|^ 2)(\|A_4A_3y\|^2- \|A_{3}^2y\|^2)-(\|A_4A_3y\|^2-\|A_3y\|^ 2 )^2 \vspace{6pt}\\
       & = -\|A_4A_3y\|^2\|(A_{3}^2-I)y\|^2 +\|y\|^ 2\|A_{3}^2y\|^2-\|A_{3}y\|^4 \vspace{6pt}\\
       & \ge 0.
\end{array} $$
From   $A_3\ge I$, we derive  $\sigma(A_3) =\sigma_{ap}(A_3) \subseteq[1,\|A_3\|]$ and $\|A_3\|\in \sigma_{ap}(A_3)$. \ Hence, there exists $y_n$ unit vectors, such that $\lim\limits_{n\to \infty} (A_3-\|A_3\|)y_n =0$. \ We  have, in particular, $\lim\limits_{n\to \infty}\|A_3y_n\|=\|A_3\|$ and $ \lim\limits_{n\to \infty}\|A_3^2y_n\|=\|A_3\|^2$. \ 

By writing 
$$
1=\|y_n\|= \|(A_4A_3)^{-1}A_4A_3y_n\|\le \|(A_4A_3)^{-1}\|\|A_4A_3y_n\|,
$$
we derive 
$$
-(\|A_3\|^2-1)^2 \ge 0.
$$  
It follows  that $\|A_3\|^2 = I$, and then $\sigma(A_3)= \{1\}$. \ Finally  $A_3 = I$. \\

The proof of backward propagation runs in a similar way. \ We include it here for completeness.

Suppose $A_2=A_3=A_4=I$ and let us show that $A_1=I$. \ As above, 
we have $$ M_{n,s}  \ge 0 \Rightarrow \begin{pmatrix}
      D_{1,s}& R_{1,s}^*         & 0       \\
    R_{1,s} & D_{2,s} &R_{2,s}^*   \\
    0              &R_{2,s} &D_{3,s}\\
\end{pmatrix}\ge 0. $$
Using vectors of the form $(A_1x,x, x)$, and denoting $\Gamma'= \| x\|^2-\|A_1 x\|^2 $ it follows that
\begin{center}
\tiny{$$  \begin{pmatrix}
     (1+s^2)\|A_1x\|^2-\|A_0A_1x\|^2 & s(\|A_1x\|^2-\|A_0A_1x\|^2 )        & 0       \\
  s(\|A_1x\|^2-\|A_0A_1x\|^2 )  & \Gamma'+s^2(\| x\|^2- \|A_0A_1 x\|^2)&s\Gamma'  \\
    0              &s\Gamma'  & s^2\Gamma'\\
\end{pmatrix}\ge 0. $$}
\end{center}
In particular, the determinant is positive. \ After the  column operation $C_2\to C_2-\frac{1}{s}C_3$, and after factoring out $s^2$, we obtain 
$$
p(s)=: \left|\begin{array}{ll}
     (1+s^2)\|A_1^2x\|^2-\|A_0A_1x\|^2& \|A_1x\|^2-\|A_0A_1x\|^2 \vspace{6pt}\\
  \|A_1x\|^2-\|A_0A_1x\|^2  & \| x\|^2- \|A_0A_1 x\|^2
\end{array}\right|\ge 0.
$$
Then 
$$
p(0)= (\|A_1^2x\|^2-\|A_0A_1x\|^2) (\| x\|^2- \|A_0A_1 x\|^2) - (\|A_1x\|^2-\|A_0A_1x\|)^2) 
   \ge 0,$$
Expanding gives
   $$
   -\|A_0A_1x\|^2( \|A_1^2x\|^2+\| x\|^2  -2\|A_1x\|^2) +\| x\|^2\|A_1x\|^2- \|A_1x\|^4\ge 0,
   $$
   and equivalently,
    $$
    -\|A_0A_1x\|^2( \|(A_1-I) x\|^2+\| x\|^2\|A_1x\|^2- \|A_1x\|^4\ge 0,
    $$
We now use a symmetric argument: $A_1\le A_2= I$ implies   $\sigma(A_1) =\sigma_{ap}(A_1) \subseteq[\alpha, 1]$ and $\alpha = min(\sigma(A_1)) \in \sigma_{ap}(A_1)$. \ Hence, there exists $y_n$ unit vectors, such that $\lim\limits_{n\to \infty} (A_1-\alpha)y_n =0$. \ We get in particular $\lim\limits_{n\to \infty}\|A_1y_n\|=\alpha, \lim\limits_{n\to \infty}\|A_1^2y_n\|=\alpha^2$ and by writing $1=\|y_n\|= \|(A_0A_1)^{-1}A_0A_1y_n\|\le \|(A_0A_1)^{-1}\|\|A_0A_1y_n\|$, we derive   $$-(\alpha^2-1)^2 \ge 0,$$  
 The last fact gives $\sigma(A_1) =\{1\}$, and since $A_1$ is selfadjoint, this forces $A_1~=~I$. 
\end{proof}
We use Propositions \ref{propag1} and \ref{propag2} to derive: 
\begin{thm}\label{quadra} Let  $ W_{\mathcal A}$ be a quadratically hyponormal and $A_n=A_{n+1} $ for some $n\ge 1 $. \ Then $W_{\mathcal A}$ is flat.
\end{thm}

Using Theorem \ref{cubl} and the fact that cubically hyponormal operators are quadratically hyponormal, we derive the next corollary. 
 \begin{coro} 
    Let $W_{\mathcal A}$ be a cubically hyponormal matrix-valued weighted shift such that $A_{n_0}=A_{n_0+1} $ for some $n_0\ge 0$. \ Then the following statements hold.
    \begin{enumerate}
        \item If $n_0\ge 1$, we have  $A_{n}=A_{1} $ for every $n\ge 1$. \vspace{6pt}
           \item If $n_0=0$, then   $A_{n}=A_{0} $ for every $n\ge 1$.
    \end{enumerate}
\end{coro}
\section{Concluding remarks and open questions}
\begin{remark} 
\begin{itemize}
    \item It is not difficult to find a $2$--hyponormal matrix-valued weighted shift $W_{\mathcal A}$ such that $A_0=A_{1}$ is not flat. \ Using Theorem \ref{cubl}, such a shift cannot be cubically hyponormal. 
    \item In contrast with cubically hyponormal and $2$--hyponormal matrix-valued weighted shifts, we have not been able to obtain local propagation results for quadratically hyponormal matrix-valued weighted shifts. \ However, using our techniques, it is possible to obtain the forward propagation results. \ It would be of interest to show that the backward propagation phenomenon also holds.
    \item All results in this paper extend easily to operator-valued weighted shifts with algebraic weight sequences. \ It is then natural to ask if these results extend to the more general case of non-algebraic operator-valued weighted shifts.
\end{itemize}
\end{remark} 

\section{Declarations}
\subsection{Funding} 
The first-named author was partially supported by NSF Grant DMS-2247167. \ The last-named author was partially supported by the Arab Fund Foundation Fellowship Program, The Distinguished Scholar Award-File 1026.

\:::::::::::::::::::::

\subsection{Conflicts of interest/competing interests}

{\bf Non-financial interests}: \:::::.:::

\subsection*{Data availability.}
All data generated or analyzed during this study are included in this article.


\end{document}